\documentclass[10pt]{article}
\usepackage[hidelinks]{hyperref}
\usepackage{amsmath}
\usepackage{amsfonts}
\usepackage{amssymb}
\usepackage{amsthm}

\usepackage[shortlabels]{enumitem}
\usepackage{microtype}
\usepackage{fullpage}
\usepackage{mathrsfs}
\usepackage{mleftright}
\usepackage{tikz}
\mleftright

\NewDocumentCommand{\R}{}{\mathbb{R}}

\NewDocumentCommand{\N}{}{\mathbb{N}}

\NewDocumentCommand{\CSpace}{o}{C\IfValueT{#1}{(#1)}}
\NewDocumentCommand{\CinftySpace}{o}{C^\infty\IfValueT{#1}{(#1)}}
\NewDocumentCommand{\CjSpace}{m o}{C^{#1}\IfValueT{#2}{(#2)}}
\NewDocumentCommand{\CinftycptSpace}{o}{C^\infty_{\mathrm{cpt}}\IfValueT{#1}{(#1)}}

\NewDocumentCommand{\Lozo}{}{L^1([0,1])}

\NewDocumentCommand{\LpSpace}{m o}{L^{#1}\IfValueT{#2}{(#2)}}
\NewDocumentCommand{\LinftySpace}{o}{\LpSpace{\infty}[#1]}

\NewDocumentCommand{\Imaginary}{}{\mathrm{Im}\: }
\NewDocumentCommand{\Real}{}{\mathrm{Re}\: }

\NewDocumentCommand{\Gt}{}{\widetilde{G}}

\newif\ifconversationdraft
\conversationdraftfalse  

\ifconversationdraft

  \newenvironment{promptbox}
    {%
      \par\medskip
      \begingroup
      \small
      \noindent\textbf{Prompt}\par
      \smallskip
      \noindent\ignorespaces
    }
    {%
      \par
      \endgroup
      \medskip
    }

\else

  \usepackage[most]{tcolorbox}
  \usepackage{xcolor}

  \definecolor{PromptFrame}{HTML}{2F5597}
  \definecolor{PromptBack}{HTML}{F4F8FF}

  \tcbset{
    promptstyle/.style={
      enhanced,
      breakable,
      boxrule=0.5pt,
      arc=2mm,
      left=8pt,
      right=8pt,
      top=7pt,
      bottom=7pt,
      before skip=10pt,
      after skip=10pt,
      fonttitle=\bfseries,
      colback=PromptBack,
      colframe=PromptFrame,
      title={Prompt}
    }
  }

  \newtcolorbox{promptbox}{promptstyle}

\fi

\begin{document}

\newtheorem{theorem}{Theorem}[section]
\newtheorem{corollary}[theorem]{Corollary}
\newtheorem{proposition}[theorem]{Proposition}
\newtheorem{lemma}[theorem]{Lemma}
\newtheorem{conjecture}[theorem]{Conjecture}
\newtheorem{problem}[theorem]{Problem}

\theoremstyle{remark}
\newtheorem{remark}[theorem]{Remark}

\theoremstyle{definition}
\newtheorem{definition}[theorem]{Definition}

\theoremstyle{remark}
\newtheorem{example}[theorem]{Example}

\numberwithin{equation}{section}
\numberwithin{figure}{section}

\title{Non-uniqueness for a differential equation with a difference quotient}

\author{Brian Street}
\date{}

\maketitle

\begin{abstract}
  Let $f(t,x),M(t,x)\in C([0,1]^2)$ with $M(t,x)>0$. We consider differential equations of the form
\[
    \frac{\partial f}{\partial t}(t,x)=\frac{M(t,x)f(t,x)-M(t,0)f(t,0)}{x},\quad x>0.
\]
For a fixed positive weight $M$, we ask whether the condition $f(0,x)=0$ forces $f\equiv 0$. 
We show the answer is negative for smooth functions: there exist $f(t,x),M(t,x)\in C^{\infty}([0,1]^2)$ with $f(0,x)=0$, $f(t,0)\not\equiv 0$, and $M(t,x)>0$ satisfying the above equation. 
However, we show that for a Borel-Laplace class of $M(t,x)$, 
the equation does have uniqueness. We relate this to uniqueness/non-uniqueness theorems for weighted Laplace transforms. 
A key example originated in an output by ChatGPT-5.5-Pro.
\end{abstract}

\section{Introduction}
Let \(f(t,x),M(t,x)\in \CSpace[\lbrack 0,1\rbrack^2]\)
with \(M(t,x)>0\).
Consider the following linear differential equation
for \(f(t,x)\)
\begin{equation}\label{Eqn::Intro::MainDiffEq}
    \frac{\partial}{\partial t} f(t,x)
    =\frac{M(t,x)f(t,x)-M(t,0)f(t,0)}{x},\quad x>0.
\end{equation}
We investigate whether this differential equation
has uniqueness: if \(f(t,x)\) satisfies\footnote{Throughout the paper, we say a function
satisfies a differential equation like \eqref{Eqn::Intro::MainDiffEq} if the derivative
exists in the classical sense, and the equation holds.} 
\eqref{Eqn::Intro::MainDiffEq}
for a given \(M(t,x)\) and \(f(0,x)=0\), must \(f(t,x)=0\), \(\forall t,x\)?
When \(M\equiv 1\), this uniqueness is a consequence of classical uniqueness results
for the Laplace transform (see \eqref{Eqn::Intro::Inverse::LaplaceUniqueExample} and the surrounding
discussion).
We show that there exist smooth \(M\) and \(f\) such that
uniqueness does not hold; however uniqueness does hold for a large
class of \(M\).
In Section \ref{Section::InverseProblems}, we relate this to inverse problems.

The next theorem (our main theorem) contrasts lack
of uniqueness for \eqref{Eqn::Intro::MainDiffEq} for general 
\(M(t,x)\), with uniqueness for a large class of \(M(t,x)\);
the uniqueness part is a consequence of \cite[Theorem 2.5]{StreetDifferentialEquationsWithDifferenceQuotient}.

\begin{theorem}\label{Thm::Intro::MainUniqueness}
    \begin{enumerate}[(I)]
        \item\label{Item::Intro::MainUniqueness::NonUnique} There exists \(M(t,x),f(t,x)\in \CinftySpace[\lbrack 0,1 \rbrack^2]\)
            satisfying  \eqref{Eqn::Intro::MainDiffEq},
            with \(M(t,x)>0\), \(\forall t,x\),
            \(f(t,0)\not\equiv 0\),
            and \(f(0,x)=0=f(1,x)\), \(\forall x\).
        \item\label{Item::Intro::MainUniqueness::Unique} Fix \(\epsilon>0\). Suppose
        \begin{equation}\label{Eqn::Intro::LaplaceM}
            M(t,x)=\frac{1}{x}\int_0^\epsilon e^{-w/x}b(t,w)\:dw + e^{-\epsilon/x} G(t,x),
        \end{equation}
        with \(M\) extended continuously to \(x=0\) \((M(t,0)=b(t,0))\),
        where \(b\in \CjSpace{3}[\lbrack 0,1\rbrack \times \lbrack 0,\epsilon\rbrack]\),
        \(b(t,0)>0\), \(\forall t\), and \(G\in \CSpace[\lbrack 0,1\rbrack^2]\). 
        Suppose \(f\in \CSpace[\lbrack 0,1\rbrack^2]\) satisfies \eqref{Eqn::Intro::MainDiffEq},
        has \(f(0,x)=0\), and \(f(t,0)\in \CjSpace{2}[\lbrack 0,1\rbrack]\). Then \(f(t,x)=0\), \(\forall t,x\).
        Note, \(M(t,0)=b(t,0)\) is positive, 
        and therefore \(M(t,x)\) is positive for \(x\) small,
        but we do not assume positivity away from \(x=0\).
    \end{enumerate}
\end{theorem}

\begin{remark}
    The class of \(M\) given by \eqref{Eqn::Intro::LaplaceM} is rather large. 
    It clearly contains all functions which are positive at \(x=0\), smooth in \(t\), and real analytic in \(x\). 
    In fact, by taking \(b\) analytic in \(w\),
    it contains all functions which are positive at \(x=0\), smooth in \(t\), and which are Gevrey-1, Borel summable in \(x\)
    (for a background on Borel summation, see \cite{BalserFromDivergentPowerSeriesToAnalyticFunctions}).
    More surprisingly, the class contains examples with arbitrary prescribed Taylor series at \(x=0\), with positive zeroth term
    (since Taylor series in \(w\) of \(b(t,w)\) are in bijective correspondence with
    Taylor series in \(x\) of \(M(t,x)\)). The 
    \(e^{-\epsilon/x}G(t,x)\)
  term also allows arbitrary continuous modifications away from x=0, after inserting a transition region. It does not, however, contain every smooth function with 
  a prescribed Taylor series; the allowed flat remainder is restricted to be 
  \(O(e^{-\epsilon/x})\).
  %
    Theorem \ref{Thm::Intro::MainUniqueness}\ref{Item::Intro::MainUniqueness::NonUnique}
    shows that uniqueness does not hold for all smooth \(M(t,x)\).
\end{remark}

The key part of the proof of Theorem \ref{Thm::Intro::MainUniqueness}\ref{Item::Intro::MainUniqueness::NonUnique}
is the failure of a Paley--Wiener type result for the Laplace transform in a smooth setting. 
In the next theorem, we present this and contrast it with a (well-known) version that
gives uniqueness for a similarly large class as Theorem \ref{Thm::Intro::MainUniqueness}\ref{Item::Intro::MainUniqueness::Unique}. 

\begin{theorem}\label{Thm::Intro::PaleyWiener}
    \begin{enumerate}[(I)]
        \item\label{Item::Intro::PaleyWiener::FailureSmooth} There exists \(Q(s,x)\in \CinftySpace[\lbrack 0,1\rbrack^2]\)
            with \(Q(s,x)>0\), \(\forall s,x\), \(Q(s,x)=1\) for \(s\geq 1/3\),
            and
            \(b(s)\in\CinftycptSpace[\lbrack 0,1/3)]\)
            with \(b\not \equiv 0\) and
            \begin{equation}\label{Eqn::Intro::PaleyWiener::FailureSmooth::MainVanishing}
                \int_0^1 e^{-s/x}Q(s,x) b(s)\:ds =0\text{ for }x\text{ sufficiently small.}
            \end{equation}
        \item\label{Item::Intro::PaleyWiener::Unique} Let \(G(t,x)\in \LinftySpace[\lbrack 0,1\rbrack^2]\), \(b(t,s)\in \CSpace[\lbrack 0,1\rbrack^2]\)
            with \(b(t,0)>0\) and \(\partial_s b(t,s)\in \CSpace[\lbrack 0,1\rbrack^2]\). 
            Set \(Q(t,x):=\frac{1}{x}\int_0^1 e^{-s/x}b(t,s)\: ds+e^{-1/x}G(t,x)\). 
            Suppose \(a\in \Lozo\) with \(\int_0^1 e^{-t/x}Q(t,x)a(t)\: dt = O(e^{-1/x})\) for \(1/x\in \N\). 
            Then, \(a=0\) almost everywhere.
    \end{enumerate}
\end{theorem}

\begin{remark}\label{Rmk::Intro::OtherWork}
    The key new example gives
    Theorem \ref{Thm::Intro::PaleyWiener}\ref{Item::Intro::PaleyWiener::FailureSmooth}. 
    We could not find an equivalent result in the literature. The closest we found 
    was the work on non-injectivity of weighted Radon transforms
    especially Boman \cite{BomanExampleNonuniquenessForGeneralizedRadonTransform,BomanLocalNonInjectivityForWeightedRadonTransforms}.
    These ideas were furthered in \cite{GoncharovNovikovNonUniquenessForRadonTransformsWithContinuousPositiveWeights,GoncharovNovikovBreakdownInjectivityForWeightedRayTransforms}.
    The proofs in these papers have some strong similarities to the one given here; though the results are not the same,
    and the difficulties somewhat different.
\end{remark}

The main new contribution of this paper is the smooth noninjectivity construction in
Theorem \ref{Thm::Intro::PaleyWiener}\ref{Item::Intro::PaleyWiener::FailureSmooth},
and hence the nonuniqueness result in Theorem \ref{Thm::Intro::MainUniqueness}\ref{Item::Intro::MainUniqueness::NonUnique}. 
This construction
was suggested by 
ChatGPT-5.5-Pro in response to a MathOverflow question of the author from 2013
\cite{street2013quantitative}, though the proof has been  rewritten by the author,
with many needed details added, in Section \ref{Section::PW::NonUnique}.

\paragraph{Acknowledgments:}
We thank Liam Price for prompting and providing access to ChatGPT-5.5-Pro. 
In particular, he devised the precise prompt that was used.
We thank the MathBot Discord community for testing
the question on other frontier models.

\paragraph{Declaration of AI use:}
The main example used in Theorem \ref{Thm::Intro::PaleyWiener}\ref{Item::Intro::PaleyWiener::FailureSmooth}
came from a response by ChatGPT-5.5-Pro.
The proof presented here was written and checked by the author.
The author used ChatGPT-5.5 to help make Figure \ref{Figure::Countour}
and for help searching the literature; all cited references were verified by the author.

\subsection{Inverse Problems}\label{Section::InverseProblems}
We were originally motivated to study equations like \eqref{Eqn::Intro::MainDiffEq}
by inverse problems--and in fact studied nonlinear versions in \cite{StreetDifferentialEquationsWithDifferenceQuotient}. 
Uniqueness for \eqref{Eqn::Intro::MainDiffEq} can be viewed
as a prototype for a certain kind of inverse problem. 

\paragraph{Forward Problem:} Suppose we are given \(M(t,x)\), \(f(1,x)\), \(f(t,0)\),
and know that \(f\) satisfies \eqref{Eqn::Intro::MainDiffEq}. The forward problem
is to compute \(f(0,x)\) from this data.  Indeed,
\begin{equation*}
    f(0,x)=\exp\left( -\int_0^1 M(s,x)/x\:ds \right) f(1,x)
    +\frac{1}{x}\int_0^1 \exp\left( -\int_0^s M(u,x)/x\:du \right) M(s,0)f(s,0)\: ds.
\end{equation*}
In particular, in the special case \(M(s,x)=1\), this gives
\begin{equation}\label{Eqn::Intro::Forward::Laplace}
    f(0,x)=e^{-1/x}f(1,x) + \frac{1}{x}\int_0^1 e^{-s/x} f(s,0)\: ds.
\end{equation}

\paragraph{Inverse Problem:} Suppose we are given \(M(t,x)\), \(f(0,x)\), and know that
\(f\) satisfies \eqref{Eqn::Intro::MainDiffEq}.  The inverse problem is to deduce 
\(f(t,x)\) from this data; equivalently to deduce \(f(t,0)\) from this data.\footnote{\(f(0,x)\)
and \(f(t,0)\) uniquely determine \(f(t,x)\) because \eqref{Eqn::Intro::MainDiffEq} is a standard
ODE away from \(x=0\) if \(f(t,0)\) is given.}
Consider the special case when \(M(s,x)=1\). 
Using \eqref{Eqn::Intro::Forward::Laplace}, the inverse problem is to then find \(f(s,0)\) knowing
\begin{equation}\label{Eqn::Intro::Inverse::LaplaceUniqueExample}
    \frac{1}{x}\int_0^1 e^{-s/x} f(s,0)\: ds = f(0,x)-e^{-1/x}f(1,x)=f(0,x)+O(e^{-1/x}).
\end{equation}
Proposition \ref{Prop::PWUnique::StreetUnique}
(or the easier result described in Remark \ref{Rmk::PWUnique::SimonResult}) shows that \(f(0,x)\) uniquely determines \(f(s,0)\).  Thus, the inverse problem
is a generalization of these Paley--Wiener theorems for the Laplace transform.
Theorem \ref{Thm::Intro::MainUniqueness}\ref{Item::Intro::MainUniqueness::NonUnique}
shows that the general inverse problem (for general \(M(t,x)\)) does not always have uniqueness;
while Theorem \ref{Thm::Intro::MainUniqueness}\ref{Item::Intro::MainUniqueness::Unique}
shows it does have uniqueness for a large class of \(M(t,x)\).

In fact, motivated by the famous anisotropic Calder\'on problem, in \cite{StreetDifferentialEquationsWithDifferenceQuotient},
we studied differential equations of the form
\begin{equation}\label{Eqn::Intro::Inverse::NonlinearDeq}
    \frac{\partial}{\partial t} f(t,x) = \frac{F(f(t,x))-F(f(t,0))}{x},\quad x>0.
\end{equation}
We proved uniqueness (and stability and reconstruction) for such differential equations (under the assumption \(F'(f(t,0))>0\))
when \(F\) is a polynomial.  It seems likely the same methods would work when \(F\)
is real analytic; however, the case when \(F\) is smooth is perhaps most closely related
to the smooth, anisotropic Calder\'on problem. For a brief discussion on this connection,
see below.

To prove uniqueness of \eqref{Eqn::Intro::Inverse::NonlinearDeq}, a natural approach would
be as follows. Write \(F(y)-F(z)=G(y,z)(y-z)\).  Then, if \(f_1(t,x),f_2(t,x)\)
both satisfy \eqref{Eqn::Intro::Inverse::NonlinearDeq} with \(f_1(0,x)=f_2(0,x)\),
then \(g(t,x):=f_1(t,x)-f_2(t,x)\) satisfies
\begin{equation*}
    \frac{\partial}{\partial t} g(t,x)=\frac{G(f_1(t,x),f_2(t,x))g(t,x)-G(f_1(t,0),f_2(t,0))g(t,0)}{x},\quad x>0,
\end{equation*}
and \(g(0,x)=0\).  This is of the form \eqref{Eqn::Intro::MainDiffEq}. 
However, since Theorem \ref{Thm::Intro::MainUniqueness}\ref{Item::Intro::MainUniqueness::NonUnique}
shows such differential equations may have non-zero solutions, one cannot use this to conclude
\(f_1(t,x)=f_2(t,x)\).
We do not know whether the results of \cite{StreetDifferentialEquationsWithDifferenceQuotient}
can be generalized to address equations like \eqref{Eqn::Intro::Inverse::NonlinearDeq}
when \(F\) is not real analytic.  The examples given in this paper do not immediately translate to examples
for these nonlinear differential equations.

We were originally motivated to study differential equations of the form
\eqref{Eqn::Intro::Inverse::NonlinearDeq} by the famous anisotropic
Calder\'on problem \cite{CalderonOnAnInverseBoundaryValueProblem,LassasUhlmannOnDeterminingARiemannianManifoldFromTheDirichletToNeumannMap}: does the Dirichlet-to-Neumann map on a compact Riemannian manifold
with boundary, 
of dimension \(\geq 3\),
uniquely determine the Riemannian manifold structure (up to the natural boundary-fixing diffeomorphism obstruction)?
There has been a considerable amount of work on this problem. 
We do not attempt to survey this extensive literature;
see \cite{UhlmannElectricalImpedanceTomographyAndCalderonsProblem,UhlmannThirtyYearsOfCalderonsProblem}
for surveys.

Roughly speaking, the connection to the Calder\'on problem is the following.
Consider a smooth, Riemannian manifold with boundary, \(M\), and let \(t\) denote boundary
normal coordinates near the boundary. On a neighborhood of the boundary,
this decomposes the manifold into layers, one for each small \(t\geq 0\), which are
diffeomorphic to \(\partial M\), each of which inherit a metric \(g(t)\)
and Laplace-Beltrami operator \(\bigtriangleup_{g(t)}\).
Lee and Uhlmann \cite{LeeUhlmannDeterminingAnisotropicRealAnalyticConductivitiesByBoundaryMeasurements}
showed that
the Dirichlet-to-Neumann map is the initial condition for a differential equation 
similar to the following equation:
\begin{equation}\label{Eqn::Calderon::CalderonDiffEq}
    \frac{\partial}{\partial t} \Lambda(t) = \Lambda(t)^2-\bigtriangleup_{g(t)},
\end{equation}
where \(\Lambda(t)\) is a pseudodifferential operator of order \(1\) on \(\partial M\),
depending smoothly on \(t\)
(see \cite[(9.1)]{StreetDifferentialEquationsWithDifferenceQuotient} for the precise
differential equation which arises, which is a bit more complicated). If this differential equation had uniqueness, then the Dirichlet-to-Neumann
map would uniquely determine the metric near the boundary.
In fact, \cite{LeeUhlmannDeterminingAnisotropicRealAnalyticConductivitiesByBoundaryMeasurements}
use this idea in the real analytic setting, where the differential equations automatically have uniqueness.
Because the left-hand side of \eqref{Eqn::Calderon::CalderonDiffEq} is a pseudodifferential
operator of order \(1\), and the right-hand side is a pseudodifferential operator of order
\(2\), the principal symbols on the right-hand side must cancel. This shows
that the principal symbol of \(\bigtriangleup_{g(t)}\) is the square of the principal
symbol of \(\Lambda(t)\).  Thinking of the principal symbol as the limit of the symbol
as some parameter goes to \(0\), this sees \eqref{Eqn::Calderon::CalderonDiffEq}
as a non-commutative analog of \eqref{Eqn::Intro::Inverse::NonlinearDeq}. 
See \cite[Section 9]{StreetDifferentialEquationsWithDifferenceQuotient}
for a more detailed discussion and more references.
We do not know if the methods of this paper have any application to the Calder\'on problem.

\begin{remark}
    Weighted Laplace transforms with parameter-dependent weights also appear
in recent work of Feizmohammadi, Kian, and Uhlmann on partial data inverse
problems \cite{FeizmohammadiKianUhlmannPartialDataInverseProblemsForReactionDiffusionAndHeatEquations}. 
Their argument also uses the Phragmén--Lindel\"of principle and Gr\"onwall's
inequality (which both appear in the literature cited in Section \ref{Section::UniquePW}
and Gr\"onwall's inequality appears directly in Section \ref{Section::UniquePW}). 
This provides one concrete PDE inverse-problem setting in which
weighted Laplace transforms of the type considered here arise naturally.
The  counterexample in Theorem \ref{Thm::Intro::PaleyWiener}\ref{Item::Intro::PaleyWiener::FailureSmooth} concerns a different, much less structured,
smooth class of weights; and it therefore does not directly apply to such settings.
\end{remark}

\section{Paley--Wiener type theorems for the Laplace transform (Theorem \texorpdfstring{\ref{Thm::Intro::PaleyWiener}}{\ref*{Thm::Intro::PaleyWiener}})}

\subsection{Uniqueness (Proof of Theorem \texorpdfstring{\ref{Thm::Intro::PaleyWiener}\ref{Item::Intro::PaleyWiener::Unique}}{\ref*{Thm::Intro::PaleyWiener}\ref*{Item::Intro::PaleyWiener::Unique}})}
\label{Section::UniquePW}
We use the following Paley--Wiener type theorem for the Laplace transform. 
\begin{proposition}
    \label{Prop::PWUnique::StreetUnique}
    Suppose \(a(t)\in \CSpace[\lbrack 0, 1\rbrack]\) with
    \begin{equation}\label{Eqn::PWUnique::StreetUnique::AssumedBound}
        \left| n \int_0^1 e^{-nt}a(t)\: dt \right|=O(e^{-n}),\quad n\in \N.
    \end{equation}
    Then, \(a=0\).
\end{proposition}
\begin{proof}
    This follows from \cite[Proposition 8.5]{StreetDifferentialEquationsWithDifferenceQuotient};
    by multiplying both sides of the equations in that proposition by \(e^{-n}\).
\end{proof}

\begin{remark}\label{Rmk::PWUnique::SimonResult}
    This kind of uniqueness statement is closely related to classical uniqueness theorems for the Laplace transform, going back at least to Lerch \cite{LerchSurUnPointDeLaTheorieDesFonctionsGeneratricesDAbel}; see also Widder's monograph \cite{WidderLaplaceTransform}. 
    We use the formulation from \cite[Proposition 8.5]{StreetDifferentialEquationsWithDifferenceQuotient}.
    The proof of Proposition \ref{Prop::PWUnique::StreetUnique}
    is  easier (and more well-known) if one assumes  \eqref{Eqn::PWUnique::StreetUnique::AssumedBound}
    for \(n\in [0,\infty)\) instead of \(n\in \N\).
    See \cite[Lemma A.2.1]{SimonNewApproachInverseSpectralTheory} for an early proof; and
    \cite[Section 8]{StreetDifferentialEquationsWithDifferenceQuotient} for an exposition.
\end{remark}

\begin{proof}[Proof of Theorem \ref{Thm::Intro::PaleyWiener}\ref{Item::Intro::PaleyWiener::Unique}]
    Without loss of generality, \(G=0\), as it can be incorporated into the error term \(O(e^{-1/x})\).
    Define \(b(t,s)=0=a(s)\) for \(s>1\) and
    write \(\lambda=1/x\).  We have, for \(\lambda\in \N\), taking \(u=s+t\) and changing the order of the integrals,
    \begin{equation*}
    \begin{split}
         &O(e^{-\lambda})
         =\lambda \int_0^1 e^{-\lambda t} \int_0^\infty e^{-\lambda s} b(t,s)\: ds\: a(t)\: dt
         =\lambda \int_0^{\infty} e^{-\lambda u}\int_0^u b(t,u-t)a(t)\: dt\: du
         \\&=\lambda \int_0^1 e^{-\lambda u} \int_0^u b(t,u-t)a(t)\: dt\: du + O(e^{-\lambda}). 
    \end{split}
    \end{equation*}
    Since \( \int_0^u b(t,u-t)a(t)\: dt\in \CSpace[\lbrack 0, 1\rbrack]\),
    Proposition \ref{Prop::PWUnique::StreetUnique} applies to show
    \(\int_0^u b(t,u-t)a(t)\: dt=0\).  Taking \(d/du\) of this, we see
    with \(b_s(t,s)=\partial_s b(t,s)\),
    \begin{equation*}
    \begin{split}
         &0=b(u,0)a(u)+\int_0^{u} b_s(t, u-t) a(t)\: dt,\quad \text{a.e.},
    \end{split}
    \end{equation*}
    and therefore,
    \begin{equation*}
        a(u) = -\frac{1}{b(u,0)} \int_0^u b_s(t,u-t)a(t)\: dt,\quad \text{a.e.}
    \end{equation*}
    We conclude, \(|a(u)|\leq C\int_0^u |a(t)|\: dt\), almost everywhere, for some constant \(C\geq 0\). 
    Gr\"onwall's inequality implies \(a=0\).
\end{proof}


\subsection{Non-uniqueness for smooth functions (Proof of Theorem \texorpdfstring{\ref{Thm::Intro::PaleyWiener}\ref{Item::Intro::PaleyWiener::FailureSmooth}}{\ref*{Thm::Intro::PaleyWiener}\ref*{Item::Intro::PaleyWiener::FailureSmooth}})}\label{Section::PW::NonUnique}
In this section, we present the proof of Theorem \ref{Thm::Intro::PaleyWiener}\ref{Item::Intro::PaleyWiener::FailureSmooth},
which was suggested by an output of ChatGPT-5.5-Pro.

Take \(\chi\in \CinftycptSpace[\lbrack 0, 1/3)]\)
with \(\chi(s)=1\) for \(s\in [0,1/4]\). For \(s>0\), define
\(b(s):=\chi(s) e^{-1/s^2} \sin(s^{-4})\), and set \(b(0)=0\);
note that \(b\in \CinftycptSpace[\lbrack 0,1/3)]\).
Set
\begin{equation*}
    B(x):=\int_0^1 e^{-s/x}b(s)\: ds, \quad M(x):=\int_0^1 e^{-s/x}b(s)^2\: ds, \quad R(x):=\frac{1}{M(x)},
\end{equation*}
\begin{equation*}
    \eta_0(x):=-\frac{B(x)}{M(x)}=-B(x)R(x),\quad x>0,\quad x\text{ small},
\end{equation*}
and set \(\eta_0(0)=0\).
We will show \(\eta_0\) defines a smooth function on \([0,\delta_0]\) for some \(\delta_0>0\),
and vanishes to infinite order as \(x\downarrow 0\).

We use \(A_1\lesssim A_2\) to mean \(A_1\leq CA_2\) for some constant \(C\). 
    We write \(A_1\lesssim_k A_2\) to mean \(A_1\leq C_k A_2\) for some constant \(C\)
    depending on \(k\).
    We write \(A_1\approx A_2\) for \(A_1\lesssim A_2\) and \(A_2\lesssim A_1\).
    
\begin{lemma}\label{Lemma::PfSmooth::EstB}
    There exists \(c>0\) such that for every \(k\),
    \begin{equation*}
        \left| \partial_x^k B(x) \right|\lesssim_k \exp\left( -c x^{-4/5} \right).
    \end{equation*}
\end{lemma}
\begin{proof}
    We have \(B(x)=B_0(x)+B_1(x)\), where
    \begin{equation*}
      B_0(x)=\int_0^{1/4} e^{-s/x} e^{-1/s^2}\sin(s^{-4})\: ds,  
    \end{equation*}
    \begin{equation*}
        B_1(x)=\int_{1/4}^1 \chi(s)e^{-s/x} e^{-1/s^2}\sin(s^{-4})\: ds.
    \end{equation*}
    Because the integrand of \(B_1\) starts at \(1/4\), we see
    \begin{equation*}
        \left| \partial_x^k B_1(x) \right| \lesssim_k x^{-2k}e^{-1/4x} \lesssim_k e^{-1/8x}, 
    \end{equation*}
    which is better than our desired estimate for \(B\). Thus, it suffices to estimate \(B_0\). 

    Using the change of variables \(u=1/s\), we have
    \begin{equation*}
        B_0(x)=\int_4^\infty u^{-2} e^{-u^2-1/xu} sin(u^4)\: du.
    \end{equation*}
    Set \(F_x(z):=z^{-2} \exp\left( -z^2-1/(xz) +iz^4 \right)\). Then,
    \begin{equation*}
        B_0(x) = \Imaginary \int_4^\infty F_x(r)\: dr.
    \end{equation*}
    Thus, it suffices to show that there exists \(c>0\) with
    \begin{equation}\label{Eqn::ProofSmooth::EstB::STSFx}
        \left| \partial_x^k  \int_4^\infty F_x(r)\: dr \right|\lesssim_k \exp\left( -c x^{-4/5} \right).
    \end{equation}

    To estimate \(\int_4^\infty F_x(r)\: dr\), we apply a change of contour as described in
    Figure \ref{Figure::Countour}. Indeed, by Cauchy's Theorem, we have
    for \(L\) large,
    \begin{equation}\label{Eqn::ProofSmooth::EstB::ChangeContour}
        \int_4^L F_x(r)\: dr = e^{i\pi/8}\int_4^L F_x(r e^{i\pi/8})\: dr
        -\int_0^{\pi/8} F_x(L e^{i\theta}) iL e^{i\theta}\: d\theta 
        +\int_0^{\pi/8} F_x(4e^{i\theta}) i 4e^{i\theta}\: d\theta,
    \end{equation}
    where the final term comes from  $\Gamma_{\mathrm{in}}$ in Figure \ref{Figure::Countour},
    and the second to last term comes from $\Gamma_{\mathrm{out}}$.

\begin{figure}[htbp]
\centering
\begin{tikzpicture}[scale=1.2, >=stealth]

  \def\r{2.0}        
  \def\R{4.5}        
  \def\ang{22.5}     

  \draw[->] (-0.4,0) -- (\R+1.0,0) node[right] {$\Re z$};
  \draw[->] (0,-0.4) -- (0,2.4) node[above] {$\Im z$};

  \draw[thick] (\r,0) -- (\R,0);
  \draw[thick] (\R,0) arc[start angle=0,end angle=\ang,radius=\R];
  \draw[thick] ({\R*cos(\ang)},{\R*sin(\ang)}) -- ({\r*cos(\ang)},{\r*sin(\ang)});
  \draw[thick] ({\r*cos(\ang)},{\r*sin(\ang)}) arc[start angle=\ang,end angle=0,radius=\r];

    \draw[->, thick] ({0.55*\r+0.45*\R},0) -- ({0.45*\r+0.55*\R},0);

  \draw[->, thick]
    ({\R*cos(0.35*\ang)},{\R*sin(0.35*\ang)})
    arc[start angle={0.35*\ang}, end angle={0.58*\ang}, radius=\R];

  \draw[->, thick]
    ({(0.72*\R+0.28*\r)*cos(\ang)},{(0.72*\R+0.28*\r)*sin(\ang)})
    -- ({(0.52*\R+0.48*\r)*cos(\ang)},{(0.52*\R+0.48*\r)*sin(\ang)});

  \draw[->, thick]
    ({\r*cos(0.70*\ang)},{\r*sin(0.70*\ang)})
    arc[start angle={0.70*\ang}, end angle={0.42*\ang}, radius=\r];

  \fill (\r,0) circle (1.5pt);
  \fill (\R,0) circle (1.5pt);
  \fill ({\r*cos(\ang)},{\r*sin(\ang)}) circle (1.5pt);
  \fill ({\R*cos(\ang)},{\R*sin(\ang)}) circle (1.5pt);

  \node[below] at (\r,0) {$4$};
  \node[below] at (\R,0) {$L$};
  \node[left] at ({\r*cos(\ang)},{\r*sin(\ang)}) {$4e^{i\pi/8}$};
  \node[above right] at ({\R*cos(\ang)},{\R*sin(\ang)}) {$Le^{i\pi/8}$};

  \node[right] at ({\R*cos(0.60*\ang)},{\R*sin(0.60*\ang)}) {$\Gamma_{\mathrm{out}}$};
  \node[left] at ({\r*cos(0.52*\ang)},{\r*sin(0.52*\ang)}) {$\Gamma_{\mathrm{in}}$};

  \draw (1.0,0) arc[start angle=0,end angle=\ang,radius=1.0];
  \node at ({1.25*cos(0.5*\ang)},{1.25*sin(0.5*\ang)}) {$\frac{\pi}{8}$};

\end{tikzpicture}
\caption{The positively oriented contour in the sector $0 \le \arg z \le \pi/8$ with inner radius $4$ and outer radius $L$.}
\label{Figure::Countour}
\end{figure}
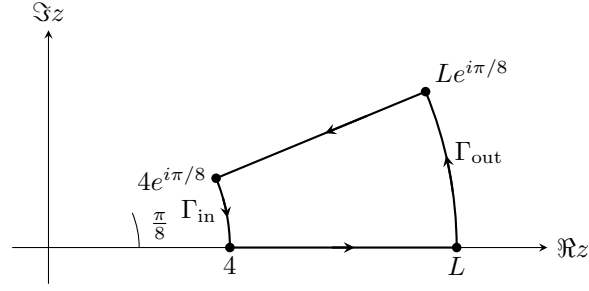

    In the remainder of the proof, we use repeatedly \(|e^z|=e^{\Real z}\).
    We have
    \begin{equation}\label{Eqn::ProofSmooth::EstimateB::Tmp1}
    \begin{split}
         &\left| \int_0^{\pi/8} F_x(L e^{i\theta}) iL e^{i\theta}\: d\theta\right|
         =\left| L^{-1} \int_0^{\pi/8} e^{-i\theta} \exp\left(  -L^2 e^{2i\theta} -\frac{1}{xL} e^{-i\theta} +iL^4 e^{4i\theta} \right)\:d\theta \right|
         \\&\leq L^{-1}\int_0^{\pi/8} \exp\left( -L^{2}\cos(2\theta) -\frac{\cos(\theta)}{xL}-L^4\sin(4\theta) \right)\:d\theta
        \\&\leq L^{-1} \int_0^{\pi/8} \exp\left( -L^2/\sqrt{2} \right)\: d\theta \xrightarrow{L\rightarrow \infty}0.
    \end{split}
    \end{equation}
    Taking \(L\rightarrow \infty\) in \eqref{Eqn::ProofSmooth::EstB::ChangeContour}
    and using \eqref{Eqn::ProofSmooth::EstimateB::Tmp1}, we
    see
    \begin{equation*}
        \int_4^\infty F_x(r)\: dr =e^{i\pi/8} \int_4^\infty F_x(re^{i\pi/8})\: dr +\int_0^{\pi/8} F_x(4e^{i\theta}) i4e^{i\theta}\: d\theta
        =:I(x)+J(x).
    \end{equation*}
    Thus, to prove \eqref{Eqn::ProofSmooth::EstB::STSFx} it suffices to show
    the next two estimates:
    \begin{equation}\label{Eqn::ProofSmooth::EstB::STSI}
        \left| \partial_x^k I(x) \right|\lesssim_k \exp\left( -cx^{-4/5} \right),
    \end{equation}
    \begin{equation}\label{Eqn::ProofSmooth::EstB::STSJ}
        \left| \partial_x^k J(x) \right|\lesssim_k \exp\left( -d/x \right),
    \end{equation}
    for some \(c,d>0\).
    
    We turn to establishing \eqref{Eqn::ProofSmooth::EstB::STSJ}. 
    We have,
    \begin{equation*}
        J(x)=\frac{i}{4}\int_0^{\pi/8} e^{-i\theta}
        \exp\left( -16 e^{2i\theta}-\frac{1}{4x} e^{-i\theta} +i 256 e^{4i\theta} \right)\: d\theta. 
    \end{equation*}
    Taking \(k\) derivatives in \(x\) of \(J(x)\) adds powers of \(x^{-1}\) of at most \(x^{-2k}\)
    and adds terms which are polynomial in \(e^{-i\theta}\). Thus, we have
    \begin{equation*}
    \begin{split}
         &\left| \partial_x^k J(x) \right|
         \lesssim_k x^{-2k} \int_0^{\pi/8} \left| 
            \exp\left( -16 e^{2i\theta} -\frac{1}{4x}e^{-i\theta}+i256 e^{4i\theta} \right)
         \right|\: d\theta
         \\&= x^{-2k}
         \int_0^{\pi/8} \exp\left( -16 \cos(2\theta)-\frac{1}{4x}\cos(\theta)-256 \sin(4\theta) \right)\: d\theta
        \\&\lesssim x^{-2k} e^{-\cos(\pi/8)/4x}
        \lesssim_k e^{-\cos(\pi/8)/8x},
        \end{split}
    \end{equation*}
    establishing \eqref{Eqn::ProofSmooth::EstB::STSJ}. 

    Finally, we turn to \eqref{Eqn::ProofSmooth::EstB::STSI}.
    Set \(\sigma:=\cos(\pi/8)>0\).
    We claim
    \begin{equation}\label{Eqn::PWSmooth::EstB::Tmp3}
        \int_4^\infty \exp\left( -r^4-\frac{\sigma}{xr} \right)\: dr \lesssim \exp\left( -c_2 x^{-4/5} \right),
    \end{equation} 
    for some \(c_2>0\) and for \(x>0\) sufficiently small.
    Set \(\psi_x(r):=\frac{1}{2}r^4+\frac{\sigma}{xr}\).
    By setting \(\psi_x'(r)=0\), we see \(\psi_x\) is minimized
    at \(r_0=\left( \sigma/2 \right)^{1/5} x^{-1/5}\).
    At this minimum, \(\psi_x(r_0)\geq \left( \sigma/2 \right)^{4/5} x^{-4/5}\);
    and therefore \(\psi_x(r)\geq \left( \sigma/2 \right)^{4/5} x^{-4/5}\), \(\forall r\).
    Using this, we have,
    \begin{equation*}
    \begin{split}
         &\int_4^\infty \exp\left( -r^4-\frac{\sigma}{xr} \right)\: dr
         =\int_4^\infty e^{-r^4/2} \exp\left( -\psi_x(r) \right)\: dr
         \\&\leq \exp\left(- \left( \sigma/2 \right)^{4/5} x^{-4/5} \right)\int_4^\infty e^{-r^4/2}\: dr
         \lesssim  \exp\left( -\left( \sigma/2 \right)^{4/5} x^{-4/5} \right),
    \end{split}
    \end{equation*}
    establishing \eqref{Eqn::PWSmooth::EstB::Tmp3}.

    Consider,
    \begin{equation*}
        I(x)=e^{i\pi/8} \int_4^{\infty} r^{-2} e^{-i\pi/4} 
        \exp\left( -r^2 e^{i\pi/4} -\frac{e^{-i\pi/8}}{xr}-r^4  \right)\: dr.
    \end{equation*}
    Taking \(k\) derivatives in \(x\) of \(I(x)\) brings down a polynomial of degree at most \(2k\) in \(x^{-1}\)
    and a polynomial in \(r^{-1}\). The polynomials in \(r^{-1}\) are bounded on the domain of integration,
    so we have, using \eqref{Eqn::PWSmooth::EstB::Tmp3},
    \begin{equation*}
    \begin{split}
         &\left| \partial_x^k I(x) \right|
         \lesssim_k x^{-2k}\int_4^{\infty} \left|\exp\left( -r^2 e^{i\pi/4} -\frac{e^{-i\pi/8}}{xr}-r^4  \right)  \right|\: dr
         \\&=x^{-2k}\int_4^{\infty}\exp\left( -r^{2} \cos(\pi/4) -\frac{\cos(\pi/8)}{xr} -r^4 \right)\: dr
         \leq x^{-2k}\int_4^{\infty}\exp\left(-\frac{\sigma}{xr} -r^4 \right)\: dr
         \\&\lesssim x^{-2k} \exp\left( -c_2 x^{-4/5} \right)
         \lesssim_k \exp\left( -c_2 x^{-4/5}/2 \right),
    \end{split}
    \end{equation*}
    establishing \eqref{Eqn::ProofSmooth::EstB::STSI} and completing the proof.
\end{proof}

\begin{lemma}\label{Lemma::ProofSmooth::LowerBoundM}
    \(\exists C\geq 1\) with \(M(x)\gtrsim e^{-C x^{-2/3}}\) for \(x>0\) small.
\end{lemma}
\begin{proof}
    We have, letting \(u=1/s\),
    \begin{equation*}
    \begin{split}
         &M(x)=\int_0^1 e^{-s/x}b(s)^2\: ds\geq \int_0^{1/4} \exp\left( -s/x-2s^{-2} \right)\sin^2(s^{-4})\: ds
         =\int_4^{\infty} u^{-2} \exp\left( -2u^2-1/xu \right) \sin^2(u^4)\: du.
    \end{split}
    \end{equation*}
    We consider only \(x\leq 1/10000\). 
    Letting \(u_0=\left( 4x \right)^{-1/3}\), we may find \(u_1\)
    such that \(|u_0-u_1|\lesssim x\) and \(\sin(u_1^4)=1\). 
    Indeed, if \(u_1=u_0+a\), then
    \begin{equation*}
        u_1^4-u_0^4=4u_0^3 a+6u_0^2 a^2+4 u_0 a^3+a^4.
    \end{equation*}
    Thus, letting \(a\) range over an interval of size \(O(x)\) around \(0\),
    lets \(u_1^4\) range over an interval around \(u_0^4\) of size at least \(2\pi\). 

    For the same reason, for \(c>0\) small, for \(u\in (u_1-cx,u_1+cx)\),
    we have \(\sin^2(u^4)\geq 1/2\). Such \(u\) satisfy \(u\approx x^{-1/3}\),
    and therefore,
    \begin{equation*}
        2 u^2 +\frac{1}{xu}\lesssim x^{-2/3}.
    \end{equation*} 
    We conclude, for some \(C\geq 1\),
    \begin{equation*}
    \begin{split}
         &M(x)\geq \int_4^\infty u^{-2} \exp\left( -2u^2-\frac{1}{xu} \right)\sin^2(u^4)\: du
         \gtrsim \int_{u_1-cx}^{u_1+cx} x^{2/3} \exp\left( -C x^{-2/3} \right)\: du
         \\&\gtrsim x^{5/3} \exp\left( -C x^{-2/3} \right)
         \gtrsim \exp\left( -2C x^{-2/3} \right),
   \end{split}
    \end{equation*}
    completing the proof with \(C\) replaced by \(2C\).
\end{proof}

\begin{lemma}\label{Lemma::ProofSmooth::UpperBoundM}
    For every \(k\),
    \begin{equation*}
        \left| \partial_x^k M(x) \right|\lesssim_k x^{-k+1}.
    \end{equation*}
\end{lemma}
\begin{proof}
    Note that
    \begin{equation*}
        \partial_x^k M(x) = x^{-k}\int_0^1 P_k(s/x)e^{-s/x}b(s)^2\: ds,
    \end{equation*}
    for some polynomial \(P_k(u)\).
    Thus,
    \begin{equation*}
    \begin{split}
         &
         \left| \partial_x^k M(x) \right|
         \lesssim x^{-k} \int_0^1 \left| P_k(s/x) \right| e^{-s/x}\: ds
         \leq x^{-k+1}\int_0^\infty \left| P_k(u) \right|e^{-u}\: du
         \lesssim_k x^{-k+1}.
    \end{split}
    \end{equation*}
\end{proof}

\begin{lemma}\label{Lemma::PfSmooth::BoundR}
    For \(x\) small, we have for every \(k\), \(\exists A_k\geq 0\),
    \begin{equation*}
        \left| \partial_x^k R(x) \right|\lesssim_k e^{A_k x^{-2/3}}.
    \end{equation*}
\end{lemma}
\begin{proof}
    We prove the result by induction on \(k\). The base case, \(k=0\),
    is Lemma \ref{Lemma::ProofSmooth::LowerBoundM}. 
    We assume the result for \(k'<k\) and prove it for \(k\geq 1\).
    Since \(M(x)R(x)=1\), we have by the product rule
    \begin{equation}\label{Eqn::PfSmooth::BoundR::Tmp1}
    \begin{split}
         &\partial_x^k R(x) = -\frac{1}{M(x)} \sum_{j=1}^k \binom{k}{j} \left( \partial_x^j M(x) \right)\left( \partial_x^{k-j}R(x) \right).
    \end{split}
    \end{equation}
    Using Lemmas \ref{Lemma::ProofSmooth::LowerBoundM} and \ref{Lemma::ProofSmooth::UpperBoundM},
    and our inductive hypothesis, \eqref{Eqn::PfSmooth::BoundR::Tmp1}
    shows
    \begin{equation*}
    \begin{split}
         &\left| \partial_x^k R(x) \right|
         \lesssim_k e^{A_0 x^{-2/3}} \sum_{j=1}^k x^{-j+1} e^{A_{k-j}x^{-2/3}}
         \lesssim_k e^{A_k x^{-2/3}},
    \end{split}
    \end{equation*}
    provided \(A_k\) is sufficiently large.
\end{proof}

\begin{lemma}\label{Lemma::PfSmooth::Esteta0}
    There exists \(c>0\) such that for \(x\) small, the following holds. 
    For every \(k\),
    \begin{equation*}
        \left| \partial_x^k \eta_0(x) \right|\lesssim_k \exp\left( -cx^{-4/5} \right).
    \end{equation*}
\end{lemma}
\begin{proof}
    Letting \(c>0\) be as in Lemma \ref{Lemma::PfSmooth::EstB} and \(A_k\geq 0\)
    be as in Lemma \ref{Lemma::PfSmooth::BoundR}.  We have, by the product rule,
    \begin{equation*}
    \begin{split}
         & \left| \partial_x^k \eta_0(x) \right|
         \lesssim_k \sum_{j=0}^k \left| \partial_x^j B(x) \right| \left| \partial_{x}^{k-j}R(x) \right|
         \lesssim_k \sum_{j=0}^k e^{-c x^{-4/5}} e^{A_{k-j} x^{-2/3}}
         \lesssim_k \exp\left( -(c/2) x^{-4/5} \right),
    \end{split}
    \end{equation*}
    completing the proof with \(c\) replaced by \(c/2\).
\end{proof}

\begin{proof}[Proof of Theorem \ref{Thm::Intro::PaleyWiener}\ref{Item::Intro::PaleyWiener::FailureSmooth}]
    Let \(\delta>0\) be small to be chosen in a moment. 
    Let \(\eta_\delta\in \CinftycptSpace[\lbrack 0,1)]\)
    satisfy \(\eta_\delta(x)=\eta_0(x)\) for \(0\leq x\leq \delta\). 
    In light of Lemma \ref{Lemma::PfSmooth::Esteta0},
    by taking \(\delta>0\)
    small, we can choose \(\eta_\delta(x)\) to be as small as we wish, uniformly in \(x\).
    
    Set \(Q(s,x):=1+\eta_\delta(x)b(s)\in \CinftySpace[\lbrack 0,1\rbrack^2]\); we take \(\delta>0\) sufficiently small
    that \(Q(s,x)>0\), \(\forall s,x\).  Since \(b(s)=0\) for \(s\geq 1/3\),
    we have \(Q(s,x)=1\) for \(s\geq 1/3\). 
    All that remains is to show \eqref{Eqn::Intro::PaleyWiener::FailureSmooth::MainVanishing} holds. 
    In fact, we have for \(0<x\leq \delta\),
    \begin{equation*}
    \begin{split}
         &\int_0^1 e^{-s/x} Q(s,x)b(s)\: ds 
         =\int_0^1 e^{-s/x} b(s)\left( 1+\eta_\delta(x)b(s) \right)\: ds
         \\&=\int_0^1 e^{-s/x} b(s)\: ds +\eta_0(x) \int_0^1 e^{-s/x}b(s)^2\: ds
         =B(x)-\frac{B(x)}{M(x)}M(x)=0.
    \end{split}
    \end{equation*}
\end{proof}

\section{Proof of Theorem \texorpdfstring{\ref{Thm::Intro::MainUniqueness}}{\ref*{Thm::Intro::MainUniqueness}}}
We begin with the proof of
Theorem \ref{Thm::Intro::MainUniqueness}\ref{Item::Intro::MainUniqueness::Unique},
which
is a consequence of \cite[Theorem 2.5]{StreetDifferentialEquationsWithDifferenceQuotient}. 
Let \(f(t,x)\) be as in Theorem \ref{Thm::Intro::MainUniqueness}\ref{Item::Intro::MainUniqueness::Unique}
with \(f(0,x)=0\).  We will show \(f(t,x)=0\) \(\forall t,x\).

With \(c(t,x):=\frac{1}{x}\int_0^{\epsilon}e^{-w/x} b(t,w)\: dw\), \(c(t,0)=b(t,0)\),
\eqref{Eqn::Intro::MainDiffEq} becomes
\begin{equation}\label{Eqn::OdePf::TurnIntoStreet}
    \begin{split}
    \frac{\partial}{\partial_t}
    f(t,x)
    &=\frac{c(t,x)f(t,x)-c(t,0)f(t,0)}{x} + \frac{1}{x}e^{-\epsilon/x}G(t,x)f(t,x)
    \\&=:\frac{c(t,x)f(t,x)-c(t,0)f(t,0)}{x} + \Gt(t,x,f(t,x)),
    \end{split}
\end{equation}
where \(\Gt(t,x,y)=\frac{1}{x}e^{-\epsilon/x}G(t,x)y\).
This is of the form covered in \cite{StreetDifferentialEquationsWithDifferenceQuotient}
in the special case \(P(t,x,y,z)=c(t,x)y\); see \cite[(2.1)]{StreetDifferentialEquationsWithDifferenceQuotient}.
Note that \(d_y P(t,0,y,y)=c(t,0)=b(t,0)>0\) and for every \(0<\gamma<\epsilon\),
\begin{equation*}
    \Gt(t,x,y)=e^{-\gamma/x}\left[ \frac{1}{x}e^{-(\epsilon-\gamma)/x}G(t,x)y \right],
\end{equation*}
where the term in \([\cdot ]\) is continuous and Lipschitz in \(y\).
The remaining hypotheses of \cite[Theorem 2.5]{StreetDifferentialEquationsWithDifferenceQuotient}
are immediate; hence that theorem  
applies
(taking \(R_k(t,x)=1\))
and shows that \eqref{Eqn::OdePf::TurnIntoStreet} has uniqueness. 
Since the zero function is also a solution, we conclude
\(f(t,x)=0\), \(\forall t,x\).

In fact, some of the assumptions 
of Theorem \ref{Thm::Intro::MainUniqueness}\ref{Item::Intro::MainUniqueness::Unique} can be relaxed
slightly and \cite[Theorem 2.5]{StreetDifferentialEquationsWithDifferenceQuotient} still gives the desired
uniqueness; we refer the reader to that paper for these and other generalizations.
The notation in that paper is very similar to the notation here, and the reader
will have no trouble seeing Theorem \ref{Thm::Intro::MainUniqueness}\ref{Item::Intro::MainUniqueness::Unique}
as a direct special case.

We turn to Theorem \ref{Thm::Intro::MainUniqueness}\ref{Item::Intro::MainUniqueness::NonUnique},
which we will see is an easy consequence of Theorem \ref{Thm::Intro::PaleyWiener}\ref{Item::Intro::PaleyWiener::FailureSmooth}.
First we present a slightly easier version of 
Theorem \ref{Thm::Intro::MainUniqueness}\ref{Item::Intro::MainUniqueness::NonUnique}. 

\begin{proposition}\label{Prop::ODEPf::EasierProp}
    There exist \(\epsilon>0\) and \(M(t,x),f(t,x)\in \CinftySpace[\lbrack 0,1 \rbrack\times \lbrack 0,\epsilon\rbrack]\)
            satisfying  \eqref{Eqn::Intro::MainDiffEq},
            with \(M(t,x)>0\), \(\forall t,x\),
             \(f(0,x)=0=f(1,x)\), \(\forall x\), and \(f(t,0)\not \equiv 0\).
\end{proposition}
\begin{proof}
    Let \(Q(s,x)\in \CinftySpace[\lbrack 0,1 \rbrack^2]\) and \(b(s)\in \CinftycptSpace[\lbrack 0,1/3)]\)
    be as in Theorem \ref{Thm::Intro::PaleyWiener}\ref{Item::Intro::PaleyWiener::FailureSmooth}.
    Set \(M(s,x):=1-x\frac{\partial}{\partial s} \log Q(s,x)\in \CinftySpace[\lbrack 0,1\rbrack^2 ]\),
    so that \(M(s,x)>0\) for \(x\in [0,\epsilon_1]\) provided \(\epsilon_1>0\) is sufficiently small.
    Set \(f(t,0):=\frac{b(t)}{M(t,0)}=b(t)\) and \(f(1,x):=0\). Finally, set
    \begin{equation*}
        \begin{split}
        f(t,x)&:=\frac{1}{x}\int_t^1 \exp\left( -\int_t^s M(u,x)/x\: du \right) M(s,0)f(s,0)\: ds
        \\&=\frac{1}{xQ(t,x)}\int_{t}^1 e^{-(s-t)/x}Q(s,x)b(s)\:ds.
        \end{split}
    \end{equation*}
    Using the above formula,
    and recalling \(f(t,0)=b(t)\), it is easy to see (and standard) that \(f\in \CinftySpace[\lbrack 0,1\rbrack \times \lbrack 0,\epsilon_1\rbrack ]\)
    and satisfies \eqref{Eqn::Intro::MainDiffEq}.  We have
    \begin{equation}\label{Eqn::ODEProof::Tmp1}
    \begin{split}
         &f(0,x)=\frac{1}{x Q(0,x)}\int_0^1 e^{-s/x}Q(s,x)b(s)\: ds.
    \end{split}
    \end{equation}
    Theorem \ref{Thm::Intro::PaleyWiener}\ref{Item::Intro::PaleyWiener::FailureSmooth}
    shows that \eqref{Eqn::ODEProof::Tmp1} vanishes for \(x\) small.
    Thus, if \(\epsilon\in (0,\epsilon_1)\) is sufficiently small, 
    Theorem \ref{Thm::Intro::PaleyWiener}\ref{Item::Intro::PaleyWiener::FailureSmooth}
    shows \(f(0,x)=0\) for \(x\in [0,\epsilon]\).  The result follows.
\end{proof}

\begin{proof}[Proof of Theorem \ref{Thm::Intro::MainUniqueness}\ref{Item::Intro::MainUniqueness::NonUnique}]
    Letting \(\epsilon>0\), \(M\), and \(f\) be as in Proposition \ref{Prop::ODEPf::EasierProp},
    Theorem \ref{Thm::Intro::MainUniqueness}\ref{Item::Intro::MainUniqueness::NonUnique}
    follows with \(M\) replaced by \(\epsilon^{-1}M(t,\epsilon x)\) and \(f\)
    replaced by \(f(t,\epsilon x)\).
\end{proof}

\appendix
\section{The proof due to ChatGPT}\label{Appendix::ChatGPT}
In this arXiv only appendix, we present the original proof given by ChatGPT verbatim; we have only modified
the formatting. We add our own commentary as footnotes. 
See Remark \ref{Rmk::Intro::OtherWork}
for comparison of this proof with the literature.
In informal tests on March 27, 2026, other frontier LLMs failed to give a proof or gave incorrect proofs;
including ChatGPT-5.2-Pro, ChatGPT-5.4-Pro, Gemini 3 DeepThink, Opus 4.6-Extended, and ChatGPT-5.5-Extended Thinking.
The below ChatGPT-5.5-Pro response was generated on April 22, 2026.

The prompt and proof given here are for a slightly different result than
Theorem \ref{Thm::Intro::PaleyWiener}\ref{Item::Intro::PaleyWiener::FailureSmooth}:
the function $a(s)$ below is constructed to be equal to \(1\) at \(1\),
whereas the similar function \(b(s)\) in Theorem \ref{Thm::Intro::PaleyWiener}\ref{Item::Intro::PaleyWiener::FailureSmooth}
has compact support in \([0,1/3)\).  This is not an essential difference, and merely
required a small additional step in the below proof (namely, the introduction
of the function \(W(x)\) which we took to be zero in Theorem \ref{Thm::Intro::PaleyWiener}\ref{Item::Intro::PaleyWiener::FailureSmooth}).

\begin{promptbox}
don't search the internet.

This is a test to see how well you can craft non-trivial, novel and creative proofs given a ``analysis'' math problem. Provide a full unconditional proof or disproof of the problem.

Problem:
\begin{quote}
``Suppose \(a(s) \in C^\infty([0,1])\) and \(H(s, x) \in C^\infty([0,1] \times [0,1])\) with \(H(s,x) > 0\). \(\forall s, x \in [0,1]\). Suppose,

\[
\sup_{\lambda>1} \left| \lambda \int_0^1 \exp\left(\lambda \int_s^1 H(r, 1/\lambda) dr\right) a(s) ds \right| < \infty.
\]

Must \(a(1) = 0\)? Must \(a \equiv 0\)?''
\end{quote}

Note: Neither \(a\) nor \(H\) is assumed real analytic. If either is real analytic, then \(a \equiv 0\).

\medskip
\noindent\hrulefill

REMEMBER - this unconditional argument may require non-trivial, creative and novel elements.
\end{promptbox}

\subsection*{ChatGPT-5.5-Pro}

\emph{I'll recast the condition as a one-sided Laplace-type transform after a smooth reparameterization, focusing on the dominant left endpoint \(s = 0\) and testing whether smooth non-analytic flatness can still permit cancellation-based counterexamples.}

\medskip

\emph{(Thought for 43m 11s)}

\medskip

No. In the \(C^\infty\) category, neither conclusion is forced. In fact, one can construct \(a, H\) with \(H > 0, a(1) \neq 0\), and the integral equal to 0 for every sufficiently large \(\lambda\).

I will write \(x = 1/\lambda\). The condition becomes a large-growth cancellation condition as \(x \downarrow 0\). The counterexample uses a flat oscillatory function near \(s = 0\) to create a huge ``cancellation reservoir,'' and then uses a flat \(x\)-dependent perturbation of the phase to cancel even a nonzero endpoint contribution from \(a(1)\).

\subsubsection*{Disproof}

Choose \(\chi \in C^\infty([0, 1])\) such that
\[
\chi(s) = 1 \quad (0 \leq s \leq 1/4), \qquad \chi(s) = 0 \quad (s \geq 1/3).
\]
For \(s > 0\), define
\[
b(s) = \chi(s) e^{-1/s^2} \sin(s^{-4}),
\]
and set \(b(0) = 0\). Then \(b \in C^\infty([0, 1])\), and \(b\) is flat at 0.

Choose \(w \in C^\infty([0, 1])\) with
\[
w(s) = 0 \quad (s \leq 1/2), \qquad w(1) = 1.
\]
Now define
\[
a(s) = b(s) + w(s).
\]
Then
\[
a(1) = 1,
\]
so this will disprove even the weaker claim \(a(1) = 0\), once \(H\) is constructed.

For \(x > 0\), define
\[
B(x) = \int_0^1 e^{-s/x} b(s) \, ds,
\]
\[
M(x) = \int_0^1 e^{-s/x} b(s)^2 \, ds,
\]
and
\[
W(x) = \int_0^1 e^{-s/x} w(s) \, ds.
\]

We will use the smooth flat function
\[
\eta_0(x) = - \frac{B(x) + W(x)}{M(x)}
\]
near \(x = 0\). The key point is that \(\eta_0\) extends smoothly to \(x = 0\) with all derivatives 0. I prove that next.

\subsubsection*{The key estimate}

First, \(W\) is exponentially small at rate \(e^{-c/x}\), because \(w\) is supported away from 0. Indeed, for every \(k\),
\[
W^{(k)}(x) = O(e^{-c/x})
\]
as \(x \downarrow 0\).

The more important part is the oscillatory term \(B\). Since \(\chi = 1\) near 0, the main contribution\footnote{That \(\int_0^{1/4}\) is the main contribution (up to appropriate \(O(e^{-c/x})\) errors) is almost immediate.
See the estimates for \(B_1\) in the proof of Lemma \ref{Lemma::PfSmooth::EstB}.}
 is
\[
\int_0^{1/4} e^{-s/x} e^{-1/s^2} \sin(s^{-4}) \, ds.
\]
Write \(u = 1/s\). Up to an error \(O(e^{-c/x})\), this\footnote{This step is needlessly confusing. 
In the previous sentence, we took just the ``main contribution,'' and now it says ``Up to an error,''
but the error was when we moved to the ``main contribution.'' 
So long as we take \(R=4\),
this formula is exactly equal to the ``main contribution.''  There is no added convenience
in picking an \(R\) other than \(4\). If one wished to pick an \(R\)
bigger than \(4\), one could just change \(1/4\) to \(1/R\) when we reduced to the
``main contribution.''}
 is the imaginary part of
\[
I_+(x) = \int_R^\infty u^{-2} e^{-u^2 - 1/(xu)} e^{i u^4} \, du
\]
for some fixed \(R > 0\), plus the analogous integral\footnote{The proof does not 
use the ``analogous integral;'' see the writeup of Lemma \ref{Lemma::PfSmooth::EstB}. One
could use this analogous integral to take the imaginary part by considering
\((I_+(x)-I_{-}(x))/2i\), but this is not necessary, and is not done below.}
 with \(e^{-i u^4}\).

For \(I_+\), rotate the contour to the ray \(u = r e^{i\pi/8}\). On that ray,
\[
\text{Re}\left( -u^2 - \frac{1}{xu} + i u^4 \right) = -r^2 \cos(\pi/4) - \frac{\cos(\pi/8)}{xr} - r^4.
\]
Hence\footnote{Here the details of this change in contour have been completely skipped.
See Lemma \ref{Lemma::PfSmooth::EstB} for the needed details.  Importantly,
the term \(J(x)\) from the proof of Lemma \ref{Lemma::PfSmooth::EstB} has been completely ignored;
this term arises from the change in contour and is the contribution from 
\(\Gamma_{\mathrm{in}}\); see Figure \ref{Figure::Countour}.
Since we need to not only bound \(J(x)\), but also bound its derivatives, this term should be mentioned.}
\[
|I_+(x)| \leq C \int_R^\infty e^{-r^4 - c/(xr)} \, dr \leq C e^{-cx^{-4/5}}.
\]
The connecting arc contributes only \(O(e^{-c/x})\). The same argument, rotating to \(-\pi/8\), treats the \(e^{-i u^4}\) term. 
Differentiating in \(x\) only introduces powers\footnote{This is the right idea, 
but is sweeping many details under the rug, 
and not paying attention to how the derivatives hit the inner part of the contour. 
One needs to prove the estimate for each derivative; one cannot directly deduce
the estimate for derivatives from the estimate without derivatives. See the proof of Lemma \ref{Lemma::PfSmooth::EstB} for the needed details.} of \(x^{-1}\) and \(u^{-1}\), so for every \(k\),
\[
B^{(k)}(x) = O(e^{-cx^{-4/5}}).
\]

Now we need a lower bound for \(M\). In the region where \(\chi = 1\), using again \(u = 1/s\),
\[
M(x) \geq \int_R^\infty u^{-2} e^{-2u^2 - 1/(xu)} \sin^2(u^4) \, du.
\]
The nonoscillatory exponent
\[
2u^2 + \frac{1}{xu}
\]
is minimized at \(u_0 = (4x)^{-1/3}\), and the minimum is \(O(x^{-2/3})\). Near \(u_0\), within distance \(O(x)\), there is a point where \(\sin^2(u^4) = 1\). On a subinterval of length \(cx\), we have \(\sin^2(u^4) \geq 1/2\), while
\[
2u^2 + \frac{1}{xu} \leq Cx^{-2/3}.
\]
Therefore\footnote{This is basically correct, 
except (as seen in the proof of Lemma \ref{Lemma::ProofSmooth::LowerBoundM}) 
it would be more direct to say that this proof gives \(M(x) \geq c x^{5/3} e^{-Cx^{-2/3}}\), 
which still implies the result by changing \(C,c\).}
\[
M(x) \geq c e^{-Cx^{-2/3}}
\]
for all sufficiently small \(x > 0\).

Combining the estimates,
\[
B^{(k)}(x) = O(e^{-cx^{-4/5}}), \qquad W^{(k)}(x) = O(e^{-c/x}), \qquad M(x) \geq e^{-Cx^{-2/3}}.
\]

Because \(4/5 > 2/3\), the quotient
\[
\eta_0(x) = - \frac{B(x) + W(x)}{M(x)}
\]
is flat at 0. More explicitly,\footnote{While the above estimates give the desired estimate for \(\eta_0\),
they do not give the desired estimate for derivatives of \(\eta_0\). For the additional estimates needed,
see Lemmas \ref{Lemma::ProofSmooth::UpperBoundM} and \ref{Lemma::PfSmooth::BoundR}, which
were not mentioned in this proof.}
 for every \(N\),
\[
\eta_0^{(k)}(x) = O(x^N)
\]
as \(x \downarrow 0\). Thus \(\eta_0\) extends to a \(C^\infty\) function near 0 by setting
\[
\eta_0(0) = 0.
\]

Now choose \(\varepsilon > 0\) small and extend \(\eta_0\) to a smooth function \(\eta \in C^\infty([0, 1])\) such that
\[
\eta(x) = \eta_0(x) \quad (0 \leq x \leq \varepsilon),
\]
and such that \(\eta\) remains sufficiently small. Define
\[
q(s, x) = 1 + \eta(x) b(s).
\]

For \(\varepsilon\) small enough,
\[
q(s, x) > 0
\]
on \([0, 1]^2\). Also \(b(1) = 0\), so
\[
q(1, x) = 1.
\]
For \(0 < x \leq \varepsilon\), we have
\[
\int_0^1 e^{-s/x} a(s) q(s, x) \, ds = \int_0^1 e^{-s/x} (b + w)(1 + \eta b) \, ds.
\]
Since \(b\) and \(w\) have disjoint supports,
\[
\int_0^1 e^{-s/x} w(s) b(s) \, ds = 0.
\]
Thus
\[
\int_0^1 e^{-s/x} a(s) q(s, x) \, ds = B(x) + W(x) + \eta(x) M(x).
\]
But \(\eta = \eta_0 = -(B + W)/M\) for \(x \leq \varepsilon\), so
\[
\int_0^1 e^{-s/x} a(s) q(s, x) \, ds = 0
\]
for all sufficiently small \(x > 0\).
Now define
\[
H(s, x) = 1 - x \partial_s \log q(s, x).
\]
Because \(q > 0\) and \(q \in C^\infty([0, 1]^2)\), we have \(H \in C^\infty([0, 1]^2)\). By taking \(\varepsilon\) smaller if needed, \(x \partial_s \log q\) is uniformly small, so
\[
H(s, x) > 0
\]
on \([0, 1]^2\).
Finally, since \(q(1, x) = 1\),
\[
\int_s^1 H(r, x) \, dr = \int_s^1 (1 - x \partial_r \log q(r, x)) \, dr = 1 - s + x \log q(s, x).
\]
Therefore, for \(x \leq \varepsilon\),
\[
\int_0^1 \exp\left( \frac{1}{x} \int_s^1 H(r, x) \, dr \right) a(s) \, ds = \int_0^1 e^{(1-s)/x} q(s, x) a(s) \, ds.
\]
Hence
\[
\int_0^1 \exp\left( \frac{1}{x} \int_s^1 H(r, x) \, dr \right) a(s) \, ds = e^{1/x} \int_0^1 e^{-s/x} a(s) q(s, x) \, ds = 0.
\]
Returning to \(\lambda = 1/x\), this says that for all sufficiently large \(\lambda\),
\[
\lambda \int_0^1 \exp\left( \lambda \int_s^1 H(r, 1/\lambda) \, dr \right) a(s) \, ds = 0.
\]
For \(1 < \lambda \leq 1/\varepsilon\), the same expression is continuous on a compact interval, hence bounded. Thus
\[
\sup_{\lambda>1} \left| \lambda \int_0^1 \exp\left( \lambda \int_s^1 H(r, 1/\lambda) \, dr \right) a(s) \, ds \right| < \infty.
\]
But \(a(1) = 1\). Therefore:
\begin{quote}
\textbf{No, \(a(1)\) need not be \(0\).}
\end{quote}
And since \(a(1) = 1\), certainly
\begin{quote}
\textbf{No, \(a\) need not be identically \(0\).}
\end{quote}

The construction is genuinely \(C^\infty\)-nonanalytic: the flat oscillatory term near \(s = 0\) and the flat \(x\)-dependent correction in \(H\) are precisely what analytic hypotheses would forbid.

\bibliographystyle{amsalpha}

\bibliography{bibliography}

\center{\it{University of Wisconsin-Madison, Department of Mathematics, 480 Lincoln Dr., Madison, WI, 53706}}

\center{\it{street@math.wisc.edu}}

\center{MSC 2020: 34A55 (Primary), 34A12, 35A02, 44A10 (Secondary)}

\center{Keywords: singular differential equation; difference quotient; weighted Laplace transform; inverse problem; non-uniqueness}

\end{document}